%% file: structural.tex
\documentclass{article}
\input{macros}

\begin{document}
\title{Structural analysis of synchronization in networks of linear oscillators}
\author{S. Emre Tuna\footnote{The author is with Department of
Electrical and Electronics Engineering, Middle East Technical
University, 06800 Ankara, Turkey. Email: {\tt etuna@metu.edu.tr}}}
\maketitle

\begin{abstract}
In networks of identical linear oscillators (e.g. pendulums
undergoing small vibrations) coupled through both dissipative
connectors (e.g. dampers) and restorative connectors (e.g. springs)
the relation between asymptotic synchronization and coupling
structure is studied. Conditions on the interconnection under which
synchronization can be achieved for some selection of coupling
strengths are established. How to strengthen those conditions so
that synchronization is guaranteed for all admissible parameter
values is also presented.
\end{abstract}

\section{Introduction}

Consider the coupled array of $q\geq 2$ linear time-invariant (LTI)
oscillators
\begin{eqnarray}\label{eqn:oscillator}
M{\ddot x}_{i}+Kx_{i}+B\left(\sum_{j=1}^{q}d_{ij}({\dot y}_{i}-{\dot
y}_{j})+\sum_{j=1}^{q}r_{ij}(y_{i}-y_{j})\right)&=&0\,,\qquad
y_{i}=B^{T}x_{i}
\end{eqnarray}
for $i=1,\,2,\,\ldots,\,q$; where $x_{i}\in\Real^{n}$,
$y_{i}\in\Real$, the matrices $M,\,K\in\Real^{n\times n}$ are
symmetric positive definite, and $B\in\Real^{n\times 1}$. Note that
the special case $n=1$ describes an assembly of coupled harmonic
oscillators \cite{ren08}. The scalars $d_{ij}=d_{ji}\geq 0$ are the
dissipative coupling strengths and $r_{ij}=r_{ji}\geq 0$ the
restorative coupling strengths. (We let $d_{ii}=0$ and $r_{ii}=0$.)
Therefore the overall coupling throughout the array can be
represented by the pair of laplacian matrices
\begin{eqnarray*}
D =
\left[\begin{array}{cccc}\sum_{j}d_{1j}&-d_{12}&\cdots&-d_{1q}\\
-d_{21}&\sum_{j}d_{2j}&\cdots&-d_{2q}\\
\vdots&\vdots&\ddots&\vdots\\
-d_{q1}&-d_{q2}&\cdots&\sum_{j}d_{qj}
\end{array}\right]\quad\mbox{and}\quad
R =
\left[\begin{array}{cccc}\sum_{j}r_{1j}&-r_{12}&\cdots&-r_{1q}\\
-r_{21}&\sum_{j}r_{2j}&\cdots&-r_{2q}\\
\vdots&\vdots&\ddots&\vdots\\
-r_{q1}&-r_{q2}&\cdots&\sum_{j}r_{qj}
\end{array}\right]\,.
\end{eqnarray*}
A simple illustration of the setup~\eqref{eqn:oscillator} is shown
in Fig.~\ref{fig:LCarray}, where four fourth-order LC oscillators
are coupled via an LTI resistor (with conductance $g_{13}$) and
three LTI inductors (with inductances
$\ell_{12},\,\ell_{23},\,\ell_{34}$). The associated laplacian
matrices read
\begin{eqnarray*}
D_{1}=\left[\begin{array}{cccc}
g_{13}&0&-g_{13}&0\\
0&0&0&0\\
-g_{13}&0&g_{13}&0\\
0&0&0&0\\
\end{array}\right]\quad\mbox{and}\quad
R_{1}=\left[\begin{array}{cccc}
\ell_{12}^{-1}&-\ell_{12}^{-1}&0&0\\
-\ell_{12}^{-1}&\ell_{12}^{-1}+\ell_{23}^{-1}&-\ell_{23}^{-1}&0\\
0&-\ell_{23}^{-1}&\ell_{23}^{-1}+\ell_{34}^{-1}&-\ell_{34}^{-1}\\
0&0&-\ell_{34}^{-1}&\ell_{34}^{-1}
\end{array}\right]\,.
\end{eqnarray*}
\begin{figure}[h]
\begin{center}
\includegraphics[scale=0.45]{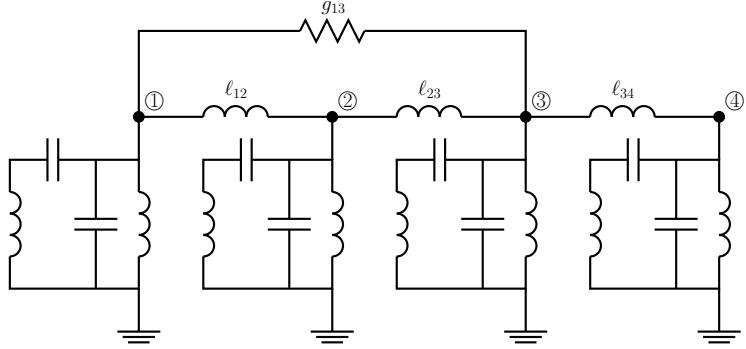}
\caption{An array of fourth-order oscillators.}\label{fig:LCarray}
\end{center}
\end{figure}

In a recent work \cite{tuna19b}, under a controllability
assumption\footnote{The triple $(M,\,K,\,B)$ satisfies: ${\rm
rank}\left[\begin{array}{c}K-\omega^{2}M\\B^{T}\end{array}\right]=n$
for all $\omega\in\Real_{>0}$.} on the triple $(M,\,K,\,B)$, it was
shown that
\begin{proposition}\label{prop:one}
If ${\rm Re}\,\lambda_{2}(D+jR)>0$\footnote{The definition of
$\lambda_{2}(\cdot)$ is given in the next section.} then (and only
then) the oscillators~\eqref{eqn:oscillator} asymptotically
synchronize, i.e., $\|x_{k}(t)-x_{\ell}(t)\|\to0$ for all
$(k,\,\ell)$ and all initial conditions.
\end{proposition}
This means that the pair $(D,\,R)$ completely characterizes the
collective behavior of the array~\eqref{eqn:oscillator} from the
synchronization point of view. Be that as it may, one would still
like to decipher (to some degree) the condition ${\rm
Re}\,\lambda_{2}(D+jR)>0$ in order to expand our understanding on
synchronization of oscillators. To this end, in this paper we study
the relation between the coupling structure and synchronization. In
other words, we investigate conditions on the underlying graphs
(disregarding the coupling strengths) that yield synchronization.
Therefore our object of inquiry here is the triple $(\V,\,\B_{\rm
d},\,\B_{\rm r})$, where
$\V=\{\nu_{1},\,\nu_{2},\,\ldots,\,\nu_{q}\}$ is the set of {\em
vertices} and the {\em edge} sets $\B_{\rm
d}=\{\{\nu_{i},\,\nu_{j}\}:d_{ij}\neq 0\}$ and $\B_{\rm
r}=\{\{\nu_{i},\,\nu_{j}\}:r_{ij}\neq 0\}$ represent the dissipative
connectors and restorative connectors, respectively. We consider two
problems; more properly speaking, two facets of one problem. To
describe those facets let us revisit our example circuitry in
Fig.~\ref{fig:LCarray} under two different scenarios. In the first
scenario, suppose that some of the coupling parameters (resistances
and inductances) are not exactly known or perhaps subject to change
due for instance to the variations of the environment temperature.
In such a case Proposition~\ref{prop:one} cannot be directly used
(because the value of ${\rm Re}\,\lambda_{2}(D+jR)$ is out of reach)
and one faces the problem of having to determine synchronization (if
possible) by looking at the interconnection $(\V,\,\B_{\rm
d},\,\B_{\rm r})$ only. For certain interconnections it is indeed
possible to be sure of asymptotic synchronization without any
knowledge (except for their sign) about the coupling strengths. Such
interconnections are said to have strong structural synchronization
(SSS) property (see Definition~\ref{def:SSS} for the formal
description) and as a part of our analysis here we investigate the
conditions guaranteeing such outcome. Another interesting situation
is that where the coupling strengths are design parameters (i.e., we
are free to choose the resistances and inductances) but the
underlying interconnection structure (which oscillator is connected
to which and with what type of connector) is predetermined and
cannot be altered. When dealing with such a case it is desirable to
know beforehand whether the given topology admits a set of parameter
values that achieves synchronization. If it does admit then we say
that it has structural synchronization (SS) property (see
Definition~\ref{def:SS}) and unearthing the conditions yielding
structural synchronization makes the other facet of the problem we
study here. Our main findings are listed below.

We establish that an interconnection is SS\footnote{Henceforth we
sometimes say an interconnection {\em is} SS to mean that it has the
SS property.} if and only if the union of the two coupling graphs
(one of them representing the dissipative coupling, the other
representing the restorative coupling) is connected and there exists
at least one dissipative connector (Theorem~\ref{thm:easy}). The
conditions guaranteeing SSS turn out to be slightly more elaborate.
We show that an interconnection is SSS if and only if it admits a
sort of flow network, where edge currents and vertex potentials obey
certain relatively nontechnical rules\footnote{These rules are
listed as (A1), (A2), (A3) later in the paper.}
(Theorem~\ref{thm:SSS2}). We also apply this result to some
benchmark topologies and obtain simple tests to check SSS property
of an interconnection when it is either a path or a cycle or a tree.

The attempts to understand behavioral properties of networks based
only on their structures have so far focused almost exclusively on
the issue of controllability \cite{xiang19,liu16}. Structural
controllability is first addressed in \cite{lin74}. The problem is
to determine whether it is possible to obtain a controllable pair
$(A,\,B)$ of matrices (representing an LTI system) under the
constraint that certain entries of the matrices must be fixed at
zero. Some time later strong structural controllability is
introduced in \cite{mayeda79}, where this time the controllability
of all admissible pairs is under consideration. Motivated by the
emergence of multi-agent systems and the need to control them, the
structural controllability theory has recently enjoyed many
important developments
\cite{rahmani09,zhang14,wang17,hou16,parlangeli12}, to name but a
few, and found interesting applications, e.g., \cite{xin18}. Despite
the rapid advances in network controllability, the relation between
coupling structure and synchronization in autonomous networks (i.e.,
those without control inputs) is relatively an unexplored field. One
of the very few works on structural synchronization is
\cite{celikovsky07}, where a pair of coupled generalized Lorenz
chaotic systems is analyzed theoretically, while the cases with
higher number of nodes is studied through numerical experiments. To
the best of our knowledge, structural synchronization analysis of
coupled LTI oscillators through (generalized) graphs $(\V,\,\B_{\rm
d},\,\B_{\rm r})$ with two different edge sets $\B_{\rm d},\,\B_{\rm
r}$ is absent from the current literature. Our contribution here is
therefore intended to be twofold: (i) bringing this void to the
attention of researchers and (ii) partially filling it by some
preliminary results.

\section{Problem statement}

We introduce some notation first. Let $e_{k}\in\Real^{q}$ be the
unit vector whose $k$th entry is 1, i.e., $e_{k}$ is the $k$th
column of the identity matrix $I$. We let $\one_{q}\in\Real^{q}$
denote the vector of all ones. A {\em graph} is a pair of sets
$(\V,\,\B)$, where $\V=\{\nu_{1},\,\nu_{2},\,\ldots,\,\nu_{q}\}$ is
the (nonempty) set of vertices (nodes) and
$\B=\{\beta_{1},\,\beta_{2},\,\ldots,\,\beta_{p}\}$ is the set of
edges (branches), where each edge is an (unordered) pair
$\{\nu_{k},\,\nu_{\ell}\}\subset\V$ of distinct vertices. The graph
$(\V,\,\B)$ can be represented by its {\em incidence matrix}
$G\in\Real^{q\times p}$ whose $i$th column is $g_{i}=e_{k}-e_{\ell}$
(or $g_{i}=e_{\ell}-e_{k}$) whenever
$\{\nu_{k},\,\nu_{\ell}\}=\beta_{i}\in\B$. Note that when $\B$ is
empty, $G$ becomes the $q$-by-0 {\em empty matrix} which (by
definition) satisfies: ${\rm range}\,G=\{0\}\subset\Real^{q}$ and
${\rm null}\,G^{T}=\Real^{q}$. Two distinct vertices
$\nu_{i},\,\nu_{j}\in\V$ are said to be {\em connected} if
$e_{i}-e_{j}\in{\rm range}\,G$. The graph $(\V,\,\B)$ is said to be
{\em connected} if every pair of distinct vertices is connected, or,
equivalently, if ${\rm null}\,G^{T}={\rm span}\,\{\one_{q}\}$. The
connected subgraphs
$(\V_{1},\,\B_{1}),\,(\V_{2},\,\B_{2}),\,\ldots,\,(\V_{c},\,\B_{c})$
are said to be the {\em components} of the graph $(\V,\,\B)$ if the
edge sets $\B_{i}$ (some of which may be empty) satisfy
$\bigcup_{i=1}^{c}\B_{i}=\B$ and the pairwise disjoint vertex sets
$\V_{i}$ (some of which may be singleton) satisfy
$\bigcup_{i=1}^{c}\V_{i}=\V$. Note that a connected graph has a
single component: itself. We also note that the number of components
satisfy $c={\rm dim}\,{\rm null}\,G^{T}$. A {\em laplacian matrix}
$L\in\Real^{q\times q}$ associated to the graph $(\V,\,\B)$ has the
form $L=G\Lambda G^{T}$, where the $p\times p$ diagonal matrix
$\Lambda={\rm diag}\,(w_{1},\,w_{2},\,\ldots,\,w_{p})$ stores the
edge weights $w_{i}>0$. If $\B=\emptyset$ then we let $L=0$. Note
that every laplacian matrix $L$ is symmetric positive semidefinite
and that ${\rm null}\,L={\rm null}\,G^{T}$. In particular,
$L\one_{q}=0$ since $G^{T}\one_{q}=0$. The set of all laplacian
matrices associated to the pair $(\V,\,\B)$ is denoted by ${\rm
lap}\,(\V,\,\B)$. A pair of graphs $[(\V,\,\B_{1}),\,(\V,\,\B_{2})]$
that share the same vertex set $\V$ defines an {\em
interconnection}, which we denote by the triple
$(\V,\,\B_{1},\,\B_{2})$. Given $X\in\Complex^{q\times q}$, we let
$\lambda_{k}(X)$ denote the $k$th smallest eigenvalue of $X$ with
respect to the real part. That is, ${\rm Re}\,\lambda_{1}(X)\leq{\rm
Re}\,\lambda_{2}(X)\leq\cdots\leq{\rm Re}\,\lambda_{q}(X)$. ${\rm
sgn}(\cdot)$ denotes the sign function. For a vector
$\eta=[\eta_{1}\ \eta_{2}\ \cdots\ \eta_{q}]^{T}\in\Real^{q}$ we let
${\rm sgn}(\eta)=[{\rm sgn}(\eta_{1})\ {\rm sgn}(\eta_{2})\ \cdots\
{\rm sgn}(\eta_{q})]^{T}$. Occasionally we write $\eta\equiv\xi$ in
lieu of ${\rm sgn}(\eta)={\rm sgn}(\xi)$. The following pair of
definitions (based on Proposition~\ref{prop:one}) is the workhorse
of our analysis.

\begin{definition}\label{def:SS}
An interconnection $(\V,\,\B_{\rm d},\,\B_{\rm r})$ is said to have
the {\em structural synchronization (SS)} property if ${\rm
Re}\,\lambda_{2}(D+jR)>0$ for some laplacian matrices $D\in{\rm
lap}(\V,\,\B_{\rm d})$ and $R\in{\rm lap}(\V,\,\B_{\rm r})$.
\end{definition}

\begin{definition}\label{def:SSS}
An interconnection $(\V,\,\B_{\rm d},\,\B_{\rm r})$ is said to have
the {\em strong structural synchronization (SSS)} property if ${\rm
Re}\,\lambda_{2}(D+jR)>0$ for all laplacian matrices $D\in{\rm
lap}(\V,\,\B_{\rm d})$ and $R\in{\rm lap}(\V,\,\B_{\rm r})$.
\end{definition}

We can now state the problems we study in this paper. Problem 1:
{\em Find conditions on the interconnection $(\V,\,\B_{\rm
d},\,\B_{\rm r})$ that guarantee structural synchronization
property.} Problem 2: {\em Find conditions on the interconnection
$(\V,\,\B_{\rm d},\,\B_{\rm r})$ that guarantee strong structural
synchronization property.} The solution of the second problem partly
depends on that of Problem~1. We therefore study structural
synchronization first.

\section{Structural synchronization}

It turns out to be very easy to check whether a given
interconnection is SS or not:

\begin{theorem}\label{thm:easy}
An interconnection $(\V,\,\B_{\rm d},\,\B_{\rm r})$ is SS if and
only if the graph $(\V,\,\B_{\rm d}\cup\B_{\rm r})$ is connected and
$\B_{\rm d}$ is nonempty.
\end{theorem}

Throughout the rest of this section we will either be proving this
result or be busy developing tools (three lemmas) for the proof.
Given the simplicity of the statement to be proven, our
demonstration seems to be unnecessarily lengthy. It is not unlikely
that there is a much shorter (and elegant) proof, but we have so far
been unable to discover it.

\begin{lemma}\label{lem:five}
Let $(\V,\,\B)$ be a connected graph. Then there exists a laplacian
$L\in{\rm lap}\,(\V,\,\B)$ such that no eigenvector of $L$ has a
zero entry.
\end{lemma}

\begin{proof}
We prove by induction. Given some connected graph
$\Gamma=(\V,\,\B)$, suppose $L\in{\rm
lap}\,(\V,\,\B)\subset\Real^{q\times q}$ has no eigenvector with
zero entry. This means (being real and symmetric) $L$ has distinct
eigenvalues, which we denote by
$\sigma_{1},\,\sigma_{2},\,\ldots,\,\sigma_{q}$. Moreover, the
corresponding unit eigenvectors
$u_{1},\,u_{2},\,\ldots,\,u_{q}\in\Real^{q}$ are pairwise
orthogonal. Let us define the positive constants $c_{1}$ and $c_{2}$
as
\begin{eqnarray*}
c_{1}&:=&\min_{i\neq j}|\sigma_{i}-\sigma_{j}|\,,\\
c_{2}&:=&\min_{i,j}|e_{i}^{T}u_{j}|\,.
\end{eqnarray*}
Let $\V=\{\nu_{1},\,\nu_{2},\,\ldots,\,\nu_{q}\}$. Now we augment
the graph $\Gamma$ by adding a new edge
$\{\nu_{k},\,\nu_{\ell}\}\notin\B$. We consider two possibilities
that preserve connectedness.

{\em Case~1: $\nu_{k},\,\nu_{\ell}\in\V$.} In this case the new edge
is between two already existing vertices. Let
$\B^{+}=\B\cup\{\{\nu_{k},\,\nu_{\ell}\}\}$. The new graph
$\Gamma^{+}=(\V,\,\B^{+})$ clearly is connected. Define
$b\in\Real^{q}$ as $b=e_{k}-e_{\ell}$. Note that $\|bb^{T}\|=2$,
where we work with the (induced) 2-norm. For the new edge choose now
some weight $w>0$ satisfying
\begin{eqnarray}\label{eqn:w}
w<\frac{c_{1}c_{2}}{8\sqrt{1+c_{2}^{2}}}
\end{eqnarray}
and construct the laplacian
\begin{eqnarray*}
L_{1}=L+w bb^{T}\,.
\end{eqnarray*}
Note that $L_{1}\in{\rm lap}\,(\V,\,\B^{+})$. Let now
$u\in\Real^{q}$ be a unit eigenvector of $L_{1}$. By
\cite[Cor.~8.1.6]{golub96} we have $L_{1}u=(\sigma+h)u$ for some
$\sigma\in\{\sigma_1,\,\sigma_2,\,\ldots,\,\sigma_q\}$ and $|h|\leq
\|wbb^{T}\|=w\|bb^{T}\|=2w$. Without loss of generality we take
$\sigma=\sigma_{1}$, i.e., $L_{1}u=(\sigma_{1}+h)u$. Let
$\alpha_{i}=u_{i}^{T}u$. Since $\{u_{1},\,u_{2},\,\ldots,\,u_{q}\}$
is an orthonormal basis for $\Real^{q}$ we have
$u=\sum_{i=1}^{q}\alpha_{i}u_{i}$. Moreover, $\|u\|=1$ implies
$\sum_{i=1}^{q}\alpha_{i}^{2}=1$. We can write
\begin{eqnarray*}
\left\|\sum_{i=2}^{q}\alpha_{i}(\sigma_{1}-\sigma_{i}+h)u_{i}\right\|
&=&\left\|-h\alpha_{1}u_{1}+
(\sigma_{1}+h)\sum_{i=1}^{q}\alpha_{i}u_{i}-\sum_{i=1}^{q}\alpha_{i}\sigma_{i}u_{i}\right\|\\
&=&\left\|-h\alpha_{1}u_{1}+
(\sigma_{1}+h)\sum_{i=1}^{q}\alpha_{i}u_{i}-L\sum_{i=1}^{q}\alpha_{i}u_{i}\right\|\\
&=&\left\|-h\alpha_{1}u_{1}+ (\sigma_{1}+h)u-\left[L_{1}-wbb^{T}\right]u\right\|\\
&=&\left\|-h\alpha_{1}u_{1}+ L_{1}u-L_{1}u+wbb^{T}u\right\|\\
&\leq&|h|\cdot|\alpha_{1}|\cdot\|u_{1}\|+w\left\|bb^{T}\right\|\cdot\|u\|\\
&\leq&4w
\end{eqnarray*}
since $|\alpha_{1}|\leq 1$ and $\|u_{1}\|=1$. Now, by \eqref{eqn:w}
we have $w<c_{1}/4$ yielding $|h|\leq c_{1}/2$. Therefore
\begin{eqnarray*} \sum_{i=2}^{q}\alpha_{i}^{2}
=\left\|\sum_{i=2}^{q}\alpha_{i}u_{i}\right\|^{2}
=\frac{4}{c_{1}^{2}}\left\|\sum_{i=2}^{q}\alpha_{i}\frac{c_{1}}{2}u_{i}\right\|^{2}
\leq\frac{4}{c_{1}^{2}}\left\|\sum_{i=2}^{q}\alpha_{i}(\sigma_{1}-\sigma_{i}+h)u_{i}\right\|^{2}\leq
\frac{64w^{2}}{c_{1}^{2}}\,.
\end{eqnarray*}
This allows us to write for all $j\in\{1,\,2,\,\ldots,\,q\}$
\begin{eqnarray}\label{eqn:almosthere}
|e_{j}^{T}u| &=&\left|e_{j}^{T}\sum_{i=1}^{q}\alpha_{i}u_{i}\right|
\geq\left|\alpha_{1}e_{j}^{T}u_{1}\right|-\left|e_{j}^{T}\sum_{i=2}^{q}\alpha_{i}u_{i}\right|\nonumber\\
&\geq&|\alpha_{1}|\cdot\left|e_{j}^{T}u_{1}\right|-\|e_{j}\|\cdot\left\|\sum_{i=2}^{q}\alpha_{i}u_{i}\right\|
\geq c_{2}\left(1-\sum_{i=2}^{q}\alpha_{i}^{2}\right)^{1/2}-\left(\sum_{i=2}^{q}\alpha_{i}^{2}\right)^{1/2}\nonumber\\
&\geq&c_{2}\left(1-\frac{64w^{2}}{c_{1}^{2}}\right)^{1/2}-\frac{8w}{c_{1}}\,.
\end{eqnarray}
Combining \eqref{eqn:w} and \eqref{eqn:almosthere} yields
$|e_{j}^{T}u|>0$ meaning $e_{j}^{T}u\neq 0$, i.e., the eigenvector
$u$ has no zero entry.

{\em Case~2: $\nu_{k}\in\V$, $\nu_{\ell}=\nu_{q+1}\notin\V$.} In
this case the new edge is between an existing vertex $\nu_{k}$ and a
new vertex $\nu_{q+1}$. Let $\V^{+}=\V\cup\{\nu_{q+1}\}$ and
$\B^{+}=\B\cup\{\{\nu_{k},\,\nu_{q+1}\}\}$. Since $\nu_{k}\in\V$ the
new graph $\Gamma^{+}=(\V^{+},\,\B^{+})$ inherits the connectedness
of $\Gamma$. Let $b\in\Real^{q+1}$ satisfy $b=e_{k}-e_{q+1}$
yielding $\|bb^{T}\|=2$. For the new edge choose now some weight
$w>0$ satisfying
\begin{eqnarray}\label{eqn:w2}
w<\frac{c_{1}c_{2}}{8\sqrt{1+q+c_{2}^{2}}}
\end{eqnarray}
and construct the laplacian
\begin{eqnarray*}
L_{2}=\underbrace{\left[\begin{array}{cc}L&0\\0&0\end{array}\right]}_{\displaystyle
\tilde L}+w bb^{T}\,.
\end{eqnarray*}
Note that $L_{2}\in{\rm lap}\,(\V^{+},\,\B^{+})$. Since $\tilde L$
is block diagonal, its eigenvalues are contained in the union
$\{\sigma_{1},\,\sigma_{2},\,\ldots,\,\sigma_{q}\}\cup\{0\}$. Being
a laplacian matrix, $L$ already has an eigenvalue at the origin.
Therefore $L$ and $\tilde L$ share the same eigenvalues. Let now
$u\in\Real^{q+1}$ be a unit eigenvector of $L_{2}$. We invoke once
again \cite[Cor.~8.1.6]{golub96} and write (without loss of
generality) $L_{2}u=(\sigma_{1}+h)u$ for some $|h|\leq 2w$. Let
$\alpha_{1},\,\alpha_{2},\,\ldots,\,\alpha_{q+1}\in\Real$ be such
that
\begin{eqnarray*}
u=\left[\begin{array}{c}\sum_{i=1}^{q}\alpha_{i}u_{i}\\
\alpha_{q+1}\end{array}\right]\,.
\end{eqnarray*}
Since $\|u\|=1$ we have $\sum_{i=1}^{q+1}\alpha_{i}^{2}=1$. Note
first that $\alpha_{q+1}$ cannot be zero. To see that suppose
otherwise. That is, $u = [\sum_{i=1}^{q}\alpha_{i}u_{i}^{T}\ \
0]^{T}$. We could then write
\begin{eqnarray*}
0=(\sigma_{1}+h)e_{q+1}^{T}u=e_{q+1}^{T}\left(\left[\begin{array}{cc}L&0\\0&0\end{array}\right]+w
bb^{T}\right)u=w(e_{q+1}^{T}b)(b^{T}u)
\end{eqnarray*}
which would imply (since $w>0$ and $e_{q+1}^{T}b=-1$) that
$b^{T}u=0$ yielding $e_{k}^{T}\sum_{i=1}^{q}\alpha_{i}u_{i}=0$,
i.e., the $k$th entry of the vector $\sum_{i=1}^{q}\alpha_{i}u_{i}$
is zero. But this would be a contradiction because $b^{T}u=0$
implies that $\sum_{i=1}^{q}\alpha_{i}u_{i}$ has to be an
eigenvector of $L$ (which has no eigenvector with a zero entry) as
can be seen through
\begin{eqnarray*}
\left[\begin{array}{c}(\sigma_{1}+h)\sum_{i=1}^{q}\alpha_{i}u_{i}\\
0\end{array}\right]=(\sigma_{1}+h)u=\left(\left[\begin{array}{cc}L&0\\0&0\end{array}\right]+w
bb^{T}\right)u=\left[\begin{array}{c}L\sum_{i=1}^{q}\alpha_{i}u_{i}\\
0\end{array}\right]\,.
\end{eqnarray*}
Having established $\alpha_{q+1}\neq 0$ let us continue by writing
\begin{eqnarray}\label{eqn:leq4w}
\left\|\left[\begin{array}{c}\sum_{i=2}^{q}\alpha_{i}(\sigma_{1}-\sigma_{i}+h)u_{i}\\(\sigma_{1}+h)\alpha_{q+1}\end{array}\right]\right\|
&=&\left\|-\left[\begin{array}{c}h\alpha_{1}u_{1}\\
0\end{array}\right] +
(\sigma_{1}+h)\left[\begin{array}{c}\sum_{i=1}^{q}\alpha_{i}u_{i}\\
\alpha_{q+1}\end{array}\right]
-\left[\begin{array}{c}\sum_{i=1}^{q}\alpha_{i}\sigma_{i}u_{i}\\ 0\end{array}\right]\right\|\nonumber\\
&=&\left\|-\left[\begin{array}{c}h\alpha_{1}u_{1}\\
0\end{array}\right]+
(\sigma_{1}+h)u-\tilde Lu\right\|\nonumber\\
&=&\left\| -\left[\begin{array}{c}h\alpha_{1}u_{1}\\
0\end{array}\right] + L_{2}u-\left[L_{2}-wbb^{T}\right]u\right\|\nonumber\\
&=&\left\| -\left[\begin{array}{c}h\alpha_{1}u_{1}\\
0\end{array}\right] + wbb^{T}u\right\|\nonumber\\
&\leq&|h|\cdot|\alpha_{1}|\cdot\|u_{1}\|+w\left\|bb^{T}\right\|\cdot\|u\|\nonumber\\
&\leq&4w\,.
\end{eqnarray}
We now study two subcases. {\em Subcase~2.1: $\sigma_{1}>0$.} Note
that since $0\in\{\sigma_{1},\,\sigma_{2},\,\ldots,\,\sigma_{q}\}$
we have $\sigma_{1}\geq c_{1}$. By \eqref{eqn:leq4w} we can
therefore write
\begin{eqnarray*} \sum_{i=2}^{q+1}\alpha_{i}^{2}
=\left\| \left[\begin{array}{c}\sum_{i=2}^{q}\alpha_{i}u_{i}\\
\alpha_{q+1}\end{array}\right] \right\|^{2} =\frac{4}{c_{1}^{2}}\left\| \left[\begin{array}{c}\sum_{i=2}^{q}\frac{c_{1}}{2}\alpha_{i}u_{i}\\
\frac{c_{1}}{2}\alpha_{q+1}\end{array}\right] \right\|^{2} \leq
\frac{4}{c_{1}^{2}}\left\|\left[\begin{array}{c}\sum_{i=2}^{q}\alpha_{i}(\sigma_{1}-\sigma_{i}+h)u_{i}\\(\sigma_{1}+h)\alpha_{q+1}\end{array}\right]\right\|^{2}
\leq \frac{64w^{2}}{c_{1}^{2}}\,.
\end{eqnarray*}
This lets us have for all $j\in\{1,\,2,\,\ldots,\,q\}$
\begin{eqnarray}\label{eqn:atmosphere}
|e_{j}^{T}u| &=&\left|e_{j}^{T}\sum_{i=1}^{q}\alpha_{i}u_{i}\right|
\geq\left|\alpha_{1}e_{j}^{T}u_{1}\right|-\left|e_{j}^{T}\sum_{i=2}^{q}\alpha_{i}u_{i}\right|\nonumber\\
&\geq&|\alpha_{1}|\cdot\left|e_{j}^{T}u_{1}\right|-\|e_{j}\|\cdot\left\|\sum_{i=2}^{q}\alpha_{i}u_{i}\right\|
\geq c_{2}\left(1-\sum_{i=2}^{q+1}\alpha_{i}^{2}\right)^{1/2}-\left(\sum_{i=2}^{q}\alpha_{i}^{2}\right)^{1/2}\nonumber\\
&\geq&c_{2}\left(1-\frac{64w^{2}}{c_{1}^{2}}\right)^{1/2}-\frac{8w}{c_{1}}\,.
\end{eqnarray}
Combining \eqref{eqn:w2} and \eqref{eqn:atmosphere} yields
$|e_{j}^{T}u|>0$ meaning $e_{j}^{T}u\neq 0$, i.e., the first $q$
entries of the eigenvector $u$ are nonzero. Recall that the last
entry $\alpha_{q+1}$ is nonzero, too. Therefore $u$ has no zero
entry. {\em Subcase~2.2: $\sigma_{1}=0$.} When $\sigma_{1}=0$ we
have $L_{2}u=hu$. That $L_{2}$ is symmetric positive semidefinite
implies that the eigenvalue $h$ is nonnegative. It turns out that
the two possibilities ($h>0$ or $h=0$) require slightly different
treatments. We begin by $h>0$. Recall that $L_{2}\one_{q+1}=0$.
Since $h\neq 0$ and $L_{2}$ is symmetric, the eigenvectors $u$ and
$\one_{q+1}$ have to be orthogonal. Moreover, $\sigma_{1}=0$ yields
$u_{1}=\one_{q}/\sqrt{q}$ since we know that $L\one_{q}=0$ and the
eigenvalue of $L$ at the origin cannot be repeated. Thence we can
write $\one_{q}^{T}u_{i}=0$ for $i=2,\,3,\,\ldots,\,q$. This implies
$\alpha_{q+1}^{2}=q\alpha_{1}^{2}$ because
\begin{eqnarray*}
\sqrt{q}\alpha_{1}+\alpha_{q+1}
=(\one_{q}^{T}u_{1})\alpha_{1}+\alpha_{q+1}+\sum_{i=2}^{q}\alpha_{i}(\one_{q}^{T}u_{i})
=\alpha_{q+1}+\one_{q}^{T}\sum_{i=1}^{q}\alpha_{i}u_{i}
=\one_{q+1}^{T}u =0\,.
\end{eqnarray*}
Therefore
\begin{eqnarray}\label{eqn:moveon}
\alpha_{1}^{2}=\frac{\alpha_{1}^{2}+\alpha_{q+1}^{2}}{q+1}=\frac{1-\sum_{i=2}^{q}\alpha_{i}^{2}}{q+1}\,.
\end{eqnarray}
Now, by \eqref{eqn:leq4w} we can write
\begin{eqnarray*} \sum_{i=2}^{q}\alpha_{i}^{2}
=\left\|\sum_{i=2}^{q}\alpha_{i}u_{i} \right\|^{2}
=\frac{4}{c_{1}^{2}}\left\|
\sum_{i=2}^{q}\frac{c_{1}}{2}\alpha_{i}u_{i}\right\|^{2} \leq
\frac{4}{c_{1}^{2}}\left\|\left[\begin{array}{c}\sum_{i=2}^{q}\alpha_{i}(\sigma_{1}-\sigma_{i}+h)u_{i}\\(\sigma_{1}+h)\alpha_{q+1}\end{array}\right]\right\|^{2}
\leq \frac{64w^{2}}{c_{1}^{2}}\,.
\end{eqnarray*}
This and \eqref{eqn:moveon} allow us write for all
$j\in\{1,\,2,\,\ldots,\,q\}$
\begin{eqnarray}\label{eqn:atmosthere}
|e_{j}^{T}u| &=&\left|e_{j}^{T}\sum_{i=1}^{q}\alpha_{i}u_{i}\right|
\geq\left|\alpha_{1}e_{j}^{T}u_{1}\right|-\left|e_{j}^{T}\sum_{i=2}^{q}\alpha_{i}u_{i}\right|\nonumber\\
&\geq&|\alpha_{1}|\cdot\left|e_{j}^{T}u_{1}\right|-\|e_{j}\|\cdot\left\|\sum_{i=2}^{q}\alpha_{i}u_{i}\right\|
\geq \frac{c_{2}}{\sqrt{q+1}}\left(1-\sum_{i=2}^{q}\alpha_{i}^{2}\right)^{1/2}-\left(\sum_{i=2}^{q}\alpha_{i}^{2}\right)^{1/2}\nonumber\\
&\geq&\frac{c_{2}}{\sqrt{q+1}}\left(1-\frac{64w^{2}}{c_{1}^{2}}\right)^{1/2}-\frac{8w}{c_{1}}\,.
\end{eqnarray}
Combining \eqref{eqn:w2} and \eqref{eqn:atmosthere} yields
$|e_{j}^{T}u|>0$ whence follows that the first $q$ entries of $u$
are nonzero. Since we also have $\alpha_{q+1}\neq 0$, the
eigenvector $u$ is free of any zero entry. Finally, we deal with the
other case $h=0$. That is, $L_{2}u=0$. Since the laplacian $L_{2}$
represents a connected graph, ${\rm dim}\,{\rm null}\,L_{2}=1$,
whence the eigenvector corresponding to the eigenvalue at the origin
is unique (up to a scaling). More precisely,
$u=\one_{q+1}/\sqrt{q+1}$, meaning $u$ has no zero entry.

To recapitulate, we have established the following. If for a given
connected graph $\Gamma=(\V,\,\B)$ there exists some laplacian
$L\in{\rm lap}\,(\V,\,\B)$ such that no eigenvector of $L$ has a
zero entry, then for any connected graph $\Gamma^{+}$ that can be
constructed from $\Gamma$ by adding a new edge, there must also
exist a laplacian with the same property. Since (i) any connected
graph can be gradually constructed in this edge-by-edge fashion
starting from a single edge connected graph
$\Gamma_{0}=(\{\nu_{1},\,\nu_{2}\},\,\{\{\nu_{1},\,\nu_{2}\}\})$ and
(ii) the eigenvectors of any laplacian $L_{0}\in{\rm
lap}\,\Gamma_{0}$ are
\begin{eqnarray*}
u_{1}=\left[\begin{array}{r}\frac{1}{\sqrt{2}}\\
\frac{1}{\sqrt{2}}\end{array}\right]\quad\mbox{and}\quad u_{2}=\left[\begin{array}{r}\frac{1}{\sqrt{2}}\\
-\frac{1}{\sqrt{2}}\end{array}\right]
\end{eqnarray*}
(i.e., they have no zero entries) the result follows by induction.
\end{proof}

\begin{lemma}\label{fact:one}
Let $D,\,R\in\Real^{q\times q}$ be symmetric positive semidefinite
matrices. Then the complex matrix $D+jR$ has no eigenvalue on the
open left half plane.
\end{lemma}

\begin{proof}
Let $\lambda\in\Complex$ be an eigenvalue of $D+jR$ with the
corresponding unit eigenvector $u\in\Complex^{q}$. We can write
${\rm Re}\,\lambda={\rm Re}\,(u^{*}\lambda u)={\rm
Re}\,(u^{*}[D+jR]u)={\rm Re}\,(u^{*}Du+ju^{*}Ru)=u^{*}Du\geq0$.
Hence the result.
\end{proof}

\begin{lemma}\label{fact:two}
Let $D,\,R\in\Real^{q\times q}$ be symmetric positive semidefinite
matrices satisfying $D\one_{q}=R\one_{q}=0$. Then ${\rm
Re}\,\lambda_{2}(D+jR)>0$ if and only if $R$ has no eigenvector in
the set ${\rm null}\,D\setminus{\rm span}\,\{\one_{q}\}$.
\end{lemma}

\begin{proof}
Suppose ${\rm Re}\,\lambda_{2}(D+jR)\leq 0$. By Lemma~\ref{fact:one}
this implies that $D+jR$ has at least two eigenvalues on the
imaginary axis. We already know that one of these eigenvalues is at
the origin since $[D+jR]\one_{q}=0$. Let $\lambda=j\sigma$ with
$\sigma\in\Real$ be the other eigenvalue. We claim that there exists
an eigenvector $u\notin{\rm span}\,\{\one_{q}\}$ satisfying
$[D+jR]u=j\sigma u$. This is obvious when $\sigma\neq0$. It is still
true when $\sigma=0$ (i.e., when the eigenvalue at the origin is
repeated) because otherwise $\one_{q}$ would be the only eigenvector
for the eigenvalue at the origin, which would imply the existence of
a generalized eigenvector $v$ satisfying $[D+jR]v=\one_{q}$. But,
since $D$ and $R$ are symmetric, this would result in the following
contradiction:
$q=\one_{q}^{T}\one_{q}=\one_{q}^{T}[D+jR]v=([D+jR]\one_{q})^{T}v=0$.
Now, without loss of generality let $\|u\|=1$. We can write
$j\sigma=u^{*}(j\sigma u)=u^{*}[D+jR]u=u^{*}Du+ju^{*}Ru$ which
implies $u^{*}Du=0$ and then $Du=0$ due to the fact that $D$ and $R$
are symmetric positive semidefinite. That $u$ belongs to ${\rm
null}\,D$ means that $Ru=\sigma u$, i.e., it is an eigenvector of
$R$, since $j\sigma u=[D+jR]u=jRu$. Hence we have established that
if the condition ${\rm Re}\,\lambda_{2}(D+jR)>0$ fails to hold then
$R$ must have an eigenvector in the set ${\rm null}\,D\setminus{\rm
span}\,\{\one_{q}\}$.

Now we show the other direction. Suppose $R$ has an eigenvector
$z\notin{\rm span}\,\{\one_{q}\}$ satisfying $Dz=0$. Let $\lambda$
be the corresponding eigenvalue, which is real since $R$ is
symmetric positive semidefinite. This implies $j\lambda$ is an
eigenvalue of $D+jR$ since $[D+jR]z=Dz+jRz=j\lambda z$. Then
combining $z\notin{\rm span}\,\{\one_{q}\}$ and $[D+jR]\one_{q}=0$
allows us to see that $D+jR$ has at least two eigenvalues on the
imaginary axis. This at once yields ${\rm Re}\,\lambda_{2}(D+jR)\leq
0$.
\end{proof}
\vspace{0.12in}

\noindent{\bf Proof of Theorem~\ref{thm:easy}.} We prove necessity
first. Let $(\V,\,\B_{\rm d},\,\B_{\rm r})$ be an SS
interconnection. This means that we can find laplacians $D\in{\rm
lap}\,(\V,\,\B_{\rm d})$ and $R\in{\rm lap}\,(\V,\,\B_{\rm r})$
satisfying ${\rm Re}\,\lambda_{2}(D+jR)>0$. Now, we claim that
$\B_{\rm d}$ cannot be empty, for otherwise $D=0$ and ${\rm
Re}\,\lambda_{2}(D+jR)={\rm Re}\,\lambda_{2}(jR)=0$ since all the
eigenvalues of $jR$ are purely imaginary thanks to that all the
eigenvalues of $R\in\Real^{q\times q}$ are real due to symmetry
$R=R^{T}$. Furthermore, the graph $\Gamma=(\V,\,\B_{\rm
d}\cup\B_{\rm r})$ must be connected. To see that suppose $\Gamma$
is not connected. Let $G$ be the incidence matrix representing
$\Gamma$. By construction we have $\one_{q}\in{\rm null}\,G^{T}$,
however ${\rm null}\,G^{T}\neq{\rm span}\,\{\one_{q}\}$ because
$\Gamma$ is not connected. Hence the null space of $G^{T}$ must be
of dimension two or larger. This allows us to be able to find a
vector $v\in{\rm null}\,G^{T}\setminus{\rm span}\,\{\one_{q}\}$. Let
$G_{\rm d}$ and $G_{\rm r}$ be the incidence matrices of the graphs
$(\V,\,\B_{\rm d})$ and $(\V,\,\B_{\rm r})$, respectively. Note that
by construction ${\rm null}\,G_{\rm d}^{T}={\rm null}\,D$ and ${\rm
null}\,G_{\rm r}^{T}={\rm null}\,R$. Since $G$ is the incidence
matrix of $\Gamma=(\V,\,\B_{\rm d}\cup\B_{\rm r})$ we can write
${\rm null}\,G_{\rm d}^{T}\cap{\rm null}\,G_{\rm r}^{T}={\rm
null}\,G^{T}\supset\{\one_{q},\,v\}$. Therefore
$(D+jR)\one_{q}=(D+jR)v=0$. This implies $D+jR$ has at least two
eigenvalues at the origin thanks to the linear independence of $v$
and $\one_{q}$. We therefore have to have ${\rm
Re}\,\lambda_{2}(D+jR)\leq 0$, which contradicts our starting point
${\rm Re}\,\lambda_{2}(D+jR)>0$.

We establish sufficiency through construction. Let the graph
$(\V,\,\B_{\rm d}\cup\B_{\rm r})$ be connected and $\B_{\rm d}$ be
nonempty. Consider the graph $\Gamma_{\rm r}=(\V,\,\B_{\rm r})$. We
study two cases. {\em Case~1: $\Gamma_{\rm r}$ is not connected.}
Let $c\leq q$ be the number of (connected) components of
$\Gamma_{\rm r}$. (Since $\Gamma_{\rm r}$ is not connected we have
$c\geq 2$.) Let the pairs
$(\V_{1},\,\B_{1}),\,(\V_{2},\,\B_{2}),\,\ldots,\,(\V_{c},\,\B_{c})$
denote these components. For each $i=1,\,2,\,\ldots,\,c$ choose now
some laplacian $R_{i}\in{\rm
lap}\,(\V_{i},\,\B_{i})\subset\Real^{q_{i}\times q_{i}}$ that has no
eigenvector with zero entry. Such choice exists thanks to
Lemma~\ref{lem:five}. Furthermore, let the following condition hold
on the collection $\{R_{1},\,R_{2},\,\ldots,\,R_{c}\}$: if $\lambda$
is a common eigenvalue of $R_{i}$ and $R_{j}$ ($i\neq j$) then
$\lambda=0$. We can impose such a condition (without violating the
condition on the eigenevectors) since $R_{i}\in{\rm
lap}\,(\V_{i},\,\B_{i})$ implies $\alpha R_{i}\in{\rm
lap}\,(\V_{i},\,\B_{i})$ for any positive scalar $\alpha$, which we
may use to relocate the nonzero eigenvalues of $R_{i}$ without
disturbing its eigenvectors. Let us now construct the block diagonal
matrix $R={\rm blkdiag}\,(R_{1},\,R_{2},\,\ldots,\,R_{c})$. Note
that $R\in{\rm lap}\,(\V,\,\B_{\rm r})$. Let now $z\notin{\rm
span}\,\{\one_q\}$ be an eigenvector of $R$ with the corresponding
eigenvalue $\lambda$, i.e., $Rz=\lambda z$. Since $R$ is positive
semidefinite, $\lambda\geq 0$. Let $G_{\rm d}$ be the incidence
matrix of the graph $(\V,\,\B_{\rm d})$. Choose $\Lambda_{\rm d}=I$
and construct the laplacian $D=G_{\rm d}G_{\rm d}^{T}=G_{\rm
d}\Lambda_{\rm d}G_{\rm d}^{T}\in{\rm lap}\,(\V,\,\B_{\rm d})$. We
claim that $Dz\neq0$, which is equivalent to $G_{\rm d}^{T}z\neq 0$.
To prove our claim we study two subcases. {\em Subcase~1.1:
$\lambda>0$.} Since the set of eigenvalues of the block diagonal
matrix $R$ is the union of sets of eigenvalues of its individual
blocks and no two blocks share a nonzero eigenvalue, $\lambda$ must
be an eigenvalue of one (and only one) of the blocks
$R_{1},\,R_{2},\,\ldots,\,R_{c}$. Without loss of generality let
that block be $R_{1}\in\Real^{q_{1}\times q_{1}}$. Hence the
eigenvector $z$ must be of the form $z=[z_{1}^{T}\ \ 0]^{T}$ where
$z_{1}\in\Real^{q_{1}}$ satisfies $R_{1}z_{1}=\lambda z_{1}$, i.e.,
$z_{1}$ is an eigenvector of $R_{1}$. Now, since $(\V,\,\B_{\rm
d}\cup\B_{\rm r})$ is connected there has to be an edge
$\{\nu_{i},\,\nu_{j}\}\in\B_{\rm d}$ that extends between the first
component $(\V_{1},\,\B_{1})$ and another one, i.e., the indices
$i,\,j$ should satisfy $i\in\{1,\,2,\,\ldots,\,q_{1}\}$ and
$j\notin\{1,\,2,\,\ldots,\,q_{1}\}$. This implies that the incidence
matrix $G_{\rm d}$ must have a column of the form $b=e_{i}-e_{j}$
(or $b=e_{j}-e_{i}$) with $i\in\{1,\,2,\,\ldots,\,q_{1}\}$ and
$j\notin\{1,\,2,\,\ldots,\,q_{1}\}$. Hence
$b^{T}z=(e_{i}-e_{j})^{T}[z_{1}^{T}\ \ 0]^{T}=e_{i}^{T}z_{1}\neq 0$
since $z_{1}$ has no zero entry. This yields $G_{\rm d}^{T}z\neq 0$,
which in turn yields $Dz\neq 0$. {\em Subcase~1.2: $\lambda=0$.} In
this case $z$ satisfies $Rz=0$ meaning $z\in{\rm null}\,G_{\rm
r}^{T}$, where $G_{\rm r}$ is the incidence matrix of the graph
$(\V,\,\B_{\rm r})$. To show $Dz\neq 0$ suppose otherwise, i.e.,
$Dz=0$, which implies $z\in{\rm null}\,G_{\rm d}^{T}$. Therefore
$z\in{\rm null}\,G_{\rm d}^{T}\cap{\rm null}\,G_{\rm r}^{T}={\rm
null}\,G^{T}$ where $G$ is the incidence matrix of the graph
$(\V,\,\B_{\rm d}\cup\B_{\rm r})$. Since $(\V,\,\B_{\rm
d}\cup\B_{\rm r})$ is connected, we have ${\rm null}\,G^{T}={\rm
span}\,\{\one_{q}\}$ but this implies $z\in{\rm
span}\,\{\one_{q}\}$, contradicting our earlier assumption that
$z\notin{\rm span}\,\{\one_{q}\}$. This completes the proof of our
claim $Dz\neq 0$. Hence we have established the following: $R$ can
have no eigenvector $z$ in the set ${\rm null}\,D\setminus{\rm
span}\,\{\one_{q}\}$. Lemma~\ref{fact:two} then yields that the
network $(\V,\,\B_{\rm d},\,\B_{\rm r})$ is SS.

{\em Case~2: $\Gamma_{\rm r}$ is connected.} Note that if $q=2$ then
${\rm null}\,D={\rm span}\,\{\one_{q}\}$ for all $D\in{\rm
lap}\,(\V,\,\B_{\rm d})$, yielding ${\rm null}\,D\setminus{\rm
span}\,\{\one_{q}\}=\emptyset$. The result then trivially follows by
Lemma~\ref{fact:two}. Thus we henceforth let $q\geq 3$. Choose some
edge from the set $\B_{\rm d}$. Without loss of generality let this
edge be labeled $\{\nu_{q_{1}},\,\nu_{q_{1}+1}\}$, where
$q_{1}\in\{1,\,2,\,\ldots,\,q-1\}$. Then choose some disjoint vertex
sets $\V_{1},\,\V_{2}\in\V$ and edge sets $\B_{1},\,\B_{2}\in\B$
such that the pairs $(\V_{1},\,\B_{1})$ and $(\V_{2},\,\B_{2})$ are
connected subgraphs, $\nu_{q_{1}}\in\V_{1}$,
$\nu_{q_{1}+1}\in\V_{2}$, and $\V_{1}\cup\V_{2}=\V$. Obtain now the
edge set $\B_{\rm r}^{-}:=\B_{\rm r}\setminus(\B_{1}\cup\B_{2})$.
Since $(\V,\,\B_{\rm r})$ is connected the new edge set $\B_{\rm
r}^{-}$ is a strict subset of $\B_{\rm r}$. Moreover, the graph
$(\V,\,\B_{\rm r}^{-})$ has exactly two (connected) components:
$(\V_{1},\,\B_{1})$ and $(\V_{2},\,\B_{2})$, where without loss of
generality we can let
$\V_{1}=\{\nu_{1},\,\nu_{2},\,\ldots,\,\nu_{q_{1}}\}$ and
$\V_{2}=\{\nu_{q_{1}+1},\,\nu_{q_{1}+2},\,\ldots,\,\nu_{q}\}$. Let
now $q_{2}:=q-q_{1}$ and choose some $R_{1}\in{\rm
lap}\,(\V_{1},\,\B_{1})$ and $R_{2}\in{\rm lap}\,(\V_{2},\,\B_{2})$
that have no eigenvectors with zero entry. Such
$R_{1}\in\Real^{q_{1}\times q_{1}}$ and $R_{2}\in\Real^{q_{2}\times
q_{2}}$ exist thanks to Lemma~\ref{lem:five}. Construct the $q\times
q$ laplacians $R^{-}={\rm blkdiag}\,(R_{1},\,R_{2})\in{\rm
lap}\,(\V,\,\B_{\rm r}^{-})$ and $D=G_{\rm d}G_{\rm d}^{T}\in{\rm
lap}\,(\V,\,\B_{\rm d})$, where $G_{\rm d}$ is the incidence matrix
of $(\V,\,\B_{\rm d})$. By earlier analysis we know that $R^{-}$ can
have no eigenvector in the set ${\rm null}\,D\setminus{\rm
span}\,\{\one_{q}\}$. Let
$\sigma_{1},\,\sigma_{2},\,\ldots,\,\sigma_{q_{1}}$ be the
eigenvalues of $R_{1}$ with the corresponding unit eigenvectors
$u_{1},\,u_{2},\,\ldots,\,u_{q-1}\in\Real^{q_{1}}$. Since $R_{1}$
has no eigenvector with zero entry, these eigenvalues are distinct.
Consequently, the eigenvectors are pairwise orthogonal since $R_{1}$
is symmetric. Likewise, we let
$\sigma_{q_{1}+1},\,\sigma_{q_{1}+2},\,\ldots,\,\sigma_{q}$ be the
distinct eigenvalues of $R_{2}$ with the corresponding pairwise
orthogonal unit eigenvectors
$u_{q_{1}+1},\,u_{q_{1}+2},\,\ldots,\,u_{q}\in\Real^{q_{2}}$. Also,
without loss of generality (see the explanation above) we assume
$\sigma_{i}\neq\sigma_{j}$ (when $i\neq j$) unless
$\sigma_{i}=\sigma_{j}=0$. Note that the set of eigenvalues of
$R^{-}$ is the union of sets of eigenvalues of $R_{1}$ and $R_{2}$.
Define the positive constants $c_{1},\,c_{2}$ as
\begin{eqnarray*}
c_{1}&:=&\min_{\sigma_{i}\neq\sigma_{j}}|\sigma_{i}-\sigma_{j}|\,,\\
c_{2}&:=&\min_{z\in\C}\,\|Dz\|
\end{eqnarray*}
where we let $\C=\{z\in\Real^{q}:\|z\|=1,\,R^{-}z=\sigma_{i}z\
\mbox{for some}\ i,\,\mbox{and}\ \one_{q}^{T}z=0\}$. Let $B$ be the
incidence matrix of the graph $(\V,\,\B_{\rm r}\setminus\B_{\rm
r}^{-})$. Choose some $w>0$ satisfying
\begin{eqnarray}\label{eqn:wuhuu}
w<\frac{c_{1}c_{2}}{4\|B\|^{2}\sqrt{c_{2}^{2}+\|D\|^{2}}}
\end{eqnarray}
and construct the laplacian
\begin{eqnarray*}
R=R^{-}+wBB^{T}\,.
\end{eqnarray*}
Note that $R\in{\rm lap}\,(\V,\,\B_{\rm r})$. Therefore, since
$(\V,\,\B_{\rm r})$ is connected, the eigenvalue of $R$ at the
origin is not repeated. Let now $u\in\Real^{q}$ be a unit
eigenvector of $R$ satisfying $u\notin{\rm span}\,\{\one_{q}\}$.
Being a laplacian, $R$ satisfies $R\one_{q}=0$, i.e., $\one_{q}$ is
the eigenvector for the eigenvalue at the origin. Hence
$\one_{q}^{T}u=0$ because $R\in\Real^{q\times q}$ is symmetric and
the eigenvectors $\one_{q}$ and $u$ correspond to different
eigenvalues. By \cite[Cor.~8.1.6]{golub96} we have $Ru=(\sigma+h)u$
for some
$\sigma\in\{\sigma_{1},\,\sigma_{2},\,\ldots,\,\sigma_{q}\}$ and
$|h|\leq \|wBB^{T}\|=w\|B\|^{2}$. Without loss of generality we take
$\sigma=\sigma_{1}$, i.e., $Ru=(\sigma_{1}+h)u$. Let
$\alpha_{1},\,\alpha_{2},\,\ldots,\,\alpha_{q}\in\Real$ be such that
\begin{eqnarray*}
u=\left[\begin{array}{c}\sum_{i=1}^{q_{1}}\alpha_{i}u_{i}\\
\sum_{i=q_{1}+1}^{q}\alpha_{i}u_{i}\end{array}\right]\,.
\end{eqnarray*}
Since $\|u\|=1$ we have $\sum_{i=1}^{q}\alpha_{i}^{2}=1$. We can
write
\begin{eqnarray}\label{eqn:leq5w}
\left\|\left[\begin{array}{c}\sum_{i=2}^{q_{1}}\alpha_{i}(\sigma_{1}-\sigma_{i}+h)u_{i}\\
\sum_{i=q_{1}+1}^{q}\alpha_{i}(\sigma_{1}-\sigma_{i}+h)u_{i}\end{array}\right]\right\|
&=&\left\|-\left[\begin{array}{c}h\alpha_{1}u_{1}\\
0\end{array}\right] + (\sigma_{1}+h)\left[\begin{array}{c}\sum_{i=1}^{q_{1}}\alpha_{i}u_{i}\\
\sum_{i=q_{1}+1}^{q}\alpha_{i}u_{i}\end{array}\right]
-\left[\begin{array}{c}\sum_{i=1}^{q_{1}}\alpha_{i}\sigma_{i}u_{i}\\ \sum_{i=q_{1}+1}^{q}\alpha_{i}\sigma_{i}u_{i}\end{array}\right]\right\|\nonumber\\
&=&\left\|-\left[\begin{array}{c}h\alpha_{1}u_{1}\\
0\end{array}\right]+
(\sigma_{1}+h)u-R^{-}u\right\|\nonumber\\
&=&\left\| -\left[\begin{array}{c}h\alpha_{1}u_{1}\\
0\end{array}\right] + Ru-\left[R-wBB^{T}\right]u\right\|\nonumber\\
&=&\left\| -\left[\begin{array}{c}h\alpha_{1}u_{1}\\
0\end{array}\right] + wBB^{T}u\right\|\nonumber\\
&\leq&|h|\cdot|\alpha_{1}|\cdot\|u_{1}\|+w\left\|BB^{T}\right\|\cdot\|u\|\nonumber\\
&\leq&2w\|B\|^{2}\,.
\end{eqnarray}
We now study two subcases. {\em Subcase~2.1: $\sigma_{1}>0$.} Note
that by \eqref{eqn:wuhuu} we have $w<c_{1}/(4\|B\|^{2})$ from which
follows $|h|\leq c_{1}/2$. By \eqref{eqn:leq5w} we can therefore
write
\begin{eqnarray*} \sum_{i=2}^{q}\alpha_{i}^{2}
&=&\left\| \left[\begin{array}{c}\sum_{i=2}^{q_{1}}\alpha_{i}u_{i}\\
\sum_{i=q_{1}+1}^{q}\alpha_{i}u_{i}
\end{array}\right] \right\|^{2} =\frac{4}{c_{1}^{2}}\left\| \left[\begin{array}{c}\sum_{i=2}^{q_{1}}\frac{c_{1}}{2}\alpha_{i}u_{i}\\
\sum_{i=q_{1}+1}^{q}\frac{c_{1}}{2}\alpha_{i}u_{i}\end{array}\right]
\right\|^{2} \leq \frac{4}{c_{1}^{2}}\left\|\left[\begin{array}{c}\sum_{i=2}^{q_{1}}\alpha_{i}(\sigma_{1}-\sigma_{i}+h)u_{i}\\
\sum_{i=q_{1}+1}^{q}\alpha_{i}(\sigma_{1}-\sigma_{i}+h)u_{i}\end{array}\right]\right\|^{2}\\
&\leq& \frac{16w^{2}\|B\|^{4}}{c_{1}^{2}}\,.
\end{eqnarray*}
This lets us have
\begin{eqnarray}\label{eqn:sph}
\|Du\| &=&\left\|D\left[\begin{array}{c}\alpha_{1}u_{1}\\
0\end{array}\right]+D\left[\begin{array}{c}\sum_{i=2}^{q_{1}}\alpha_{i}u_{i}\\
\sum_{i=q_{1}+1}^{q}\alpha_{i}u_{i}\end{array}\right]\right\|\nonumber\\
&\geq&|\alpha_{1}|\cdot\left\|D\left[\begin{array}{c}u_{1}\\
0\end{array}\right]\right\|-\|D\|\cdot \left\| \left[\begin{array}{c}\sum_{i=2}^{q_{1}}\alpha_{i}u_{i}\\
\sum_{i=q_{1}+1}^{q}\alpha_{i}u_{i}\end{array}\right] \right\|\nonumber\\
&\geq&c_{2}\left(1-\frac{16w^{2}\|B\|^{4}}{c_{1}^{2}}\right)^{1/2}-\|D\|\frac{4w\|B\|^{2}}{c_{1}}\,.
\end{eqnarray}
Combining \eqref{eqn:wuhuu} and \eqref{eqn:sph} yields $\|Du\|>0$,
meaning $u\notin{\rm null}\,D$. {\em Subcase~2.2: $\sigma_{1}=0$.}
When $\sigma_{1}=0$ we have $Ru=hu$. That $R$ is symmetric positive
semidefinite implies that the eigenvalue $h$ is nonnegative. Now,
$h$ cannot be zero because then we have $Ru=0$ which implies
$u\in{\rm span}\,\{\one_{q}\}$, contradicting our initial
assumption. Hence $h>0$. Note that $\sigma_{1}=0$ yields
$u_{1}=\one_{q_{1}}/\sqrt{q_{1}}$ since the laplacian $R_{1}$
represents a connected component. Now, the second block of $R^{-}$
has also a an eigenvalue at the origin which is not repeated because
$R_{2}$ too represents a connected component. Without loss of
generality we can let this eigenvalue be $\sigma_{q}=0$. This means
$u_{q}=\one_{q_{2}}/\sqrt{q_{2}}$. Observing the orthogonal
relationships $\one_{q}^{T}u=0$, $\one_{q_{1}}^{T}u_{i}=0$ for
$i=2,\,3,\,\ldots,\,q_{1}$, and $\one_{q_{2}}^{T}u_{i}=0$ for
$i=q_{1}+1,\,q_{1}+2,\,\ldots,\,q-1$; we obtain the identity
$q_{2}\alpha_{q}^{2}=q_{1}\alpha_{1}^{2}$ through
\begin{eqnarray*}
\alpha_{1}\sqrt{q_{1}}+\alpha_{q}\sqrt{q_{2}}
&=&(\one_{q_{1}}^{T}u_{1})\alpha_{1}+\sum_{i=2}^{q_{1}}(\one_{q_{1}}^{T}u_{i})\alpha_{i}
+(\one_{q_{2}}^{T}u_{q})\alpha_{q}+\sum_{i=q_{1}+1}^{q-1}(\one_{q_{2}}^{T}u_{i})\alpha_{i}\\
&=&\one_{q_{1}}^{T}\sum_{i=1}^{q_{1}}\alpha_{i}u_{i}+\one_{q_{2}}^{T}\sum_{i=q_{1}+1}^{q}\alpha_{i}u_{i}
=\one_{q}^{T}u=0\,.
\end{eqnarray*}
Therefore
\begin{eqnarray}\label{eqn:moveup}
(1+q_{1}/q_{2})\alpha_{1}^{2}=\alpha_{1}^{2}+\alpha_{q}^{2}=1-\sum_{i=2}^{q-1}\alpha_{i}^{2}\,.
\end{eqnarray}
Now, by \eqref{eqn:leq5w} we can write
\begin{eqnarray*} \sum_{i=2}^{q-1}\alpha_{i}^{2}
&=&\left\| \left[\begin{array}{c}\sum_{i=2}^{q_{1}}\alpha_{i}u_{i}\\
\sum_{i=q_{1}+1}^{q-1}\alpha_{i}u_{i}
\end{array}\right] \right\|^{2} =\frac{4}{c_{1}^{2}}\left\| \left[\begin{array}{c}\sum_{i=2}^{q_{1}}\frac{c_{1}}{2}\alpha_{i}u_{i}\\
\sum_{i=q_{1}+1}^{q-1}\frac{c_{1}}{2}\alpha_{i}u_{i}\end{array}\right]
\right\|^{2} \leq \frac{4}{c_{1}^{2}}\left\|\left[\begin{array}{c}\sum_{i=2}^{q_{1}}\alpha_{i}(\sigma_{1}-\sigma_{i}+h)u_{i}\\
\sum_{i=q_{1}+1}^{q}\alpha_{i}(\sigma_{1}-\sigma_{i}+h)u_{i}\end{array}\right]\right\|^{2}\\
&\leq& \frac{16w^{2}\|B\|^{4}}{c_{1}^{2}}\,.
\end{eqnarray*}
This, \eqref{eqn:moveup}, and the identity
$\alpha_{q}=-\alpha_{1}\sqrt{q_{1}/q_{2}}$ allow us write
\begin{eqnarray}\label{eqn:quelle}
\|Du\| &=& \left\|D\left[\begin{array}{c}\alpha_{1}u_{1}\\
\alpha_{q}u_{q}\end{array}\right]+D \left[\begin{array}{c}\sum_{i=2}^{q_{1}}\alpha_{i}u_{i}\\
\sum_{i=q_{1}+1}^{q-1}\alpha_{i}u_{i}
\end{array}\right]  \right\|\nonumber\\
&\geq&\left\|\alpha_{1}D\left[\begin{array}{c}\frac{\one_{q_{1}}}{\sqrt{q_{1}}}\\
-\frac{\sqrt{q_{1}}}{\sqrt{q_{2}}}\frac{\one_{q_{2}}}{\sqrt{q_{2}}}\end{array}\right]\right\|-\|D\|\cdot \left\| \left[\begin{array}{c}\sum_{i=2}^{q_{1}}\alpha_{i}u_{i}\\
\sum_{i=q_{1}+1}^{q-1}\alpha_{i}u_{i}
\end{array}\right] \right\|\nonumber\\
&\geq&|\alpha_{1}|c_{2}\sqrt{1+q_{1}/q_{2}}-\|D\|\left(\sum_{i=2}^{q-1}\alpha_{i}^{2}\right)^{1/2}\nonumber\\
&\geq&c_{2}\left(1-\frac{16w^{2}\|B\|^{4}}{c_{1}^{2}}\right)^{1/2}-\|D\|\frac{4w\|B\|^{2}}{c_{1}}\,.
\end{eqnarray}
Combining \eqref{eqn:wuhuu} and \eqref{eqn:quelle} yields
$\|Du\|>0$, meaning $u\notin{\rm null}\,D$. Hence, in neither of the
subcases (when $\sigma_{1}>0$ and when $\sigma_{1}=0$) the laplacian
$R$ can have an eigenvector in the set ${\rm null}\,D\setminus{\rm
span}\,\{\one_{q}\}$. Then we deduce by Lemma~\ref{fact:two} that
the interconnection $(\V,\,\B_{\rm d},\,\B_{\rm r})$ has structural
synchronization property.\hfill\null\hfill$\blacksquare$

\section{Strong structural synchronization}

It directly follows from the definitions that strong structural
synchronization implies structural synchronization. We later present
examples where an SS network fails to be SSS. That is to say, SS and
SSS are not equivalent.  This brings up the question: Under what
conditions does an SS interconnection become SSS? The below theorem
presents one such characterization.

\begin{theorem}\label{thm:SSS}
Let the triple $(\V,\,\B_{\rm d},\,\B_{\rm r})$ has structural
synchronization property. And let $G_{\rm d}$ and $G_{\rm r}$ be the
incidence matrices of the graphs $(\V,\,\B_{\rm d})$ and
$(\V,\,\B_{\rm r})$, respectively. Then the interconnection
$(\V,\,\B_{\rm d},\,\B_{\rm r})$ has strong structural
synchronization property if and only if either $\B_{\rm
r}=\emptyset$ or else the equation ${\rm sgn}(G_{\rm r}^{T}G_{\rm
r}x)={\rm sgn}(x)$ has no solution $x$ in the set ${\rm
null}\,G_{\rm d}^{T}G_{\rm r}\setminus\{0\}$.
\end{theorem}

\begin{proof}
Let us be given an SS interconnection $\G=(\V,\,\B_{\rm d},\,\B_{\rm
r})$. We consider two possibilities separately. {\em Case~1:
$\B_{\rm r}=\emptyset$.} Since $\G$ is SS, by Theorem~\ref{thm:easy}
the graph $(\V,\,\B_{\rm d}\cup\B_{\rm r})$ is connected. Therefore
the graph $(\V,\,\B_{\rm d})$ is connected because $\B_{\rm r}$ is
empty. Thence ${\rm null}\,G_{\rm d}^{T}={\rm span}\,\{\one_{q}\}$.
Let $D\in{\rm lap}\,(\V,\,\B_{\rm d})$ be an arbitrary laplacian.
Since ${\rm null}\,D={\rm null}\,G_{\rm d}^{T}$ we have ${\rm
null}\,D={\rm span}\,\{\one_{q}\}$. This yields ${\rm
null}\,D\setminus{\rm span}\,\{\one_{q}\}=\emptyset$. Then by
Lemma~\ref{fact:two} we trivially have ${\rm
Re}\,\lambda_{2}(D+jR)>0$ for all $R\in{\rm lap}\,(\V,\,\B_{\rm
r})$. That is, $\G$ is SSS.

{\em Case~2: $\B_{\rm r}\neq\emptyset$.} Suppose $\G$ is not SSS.
Then there exist two diagonal matrices $\Lambda_{\rm d}$ and
$\Lambda_{\rm r}$, both with positive diagonal entries, such that
${\rm Re}\,\lambda_{2}(D_{1}+jR_{1})\leq0$ for $D_{1}=G_{\rm
d}\Lambda_{\rm d}G_{\rm d}^{T}$ and $R_{1}=G_{\rm r}\Lambda_{\rm
r}G_{\rm r}^{T}$. This implies, by Lemma~\ref{fact:two}, that there
exists an eigenvector $u\notin{\rm span}\,\{\one_{q}\}$ of $R_{1}$
satisfying $D_{1}u=0$. Let $\lambda$ be the corresponding
eigenvalue, i.e., $R_{1}u=\lambda u$. Note first that $\lambda\geq
0$ because $R_{1}$ is symmetric positive semidefinite. Also,
$\lambda\neq 0$; for, otherwise, (i.e., if $\lambda=0$) we would
simultaneously have $R_{1}u=0$ and $D_{1}u=0$, which would imply
$u\in{\rm null}\,G_{\rm r}^{T}\cap{\rm null}\,G_{\rm d}^{T}$. But
since $\G$ is SS, the graph $(\V,\,\B_{\rm d}\cup\B_{\rm r})$ is
connected by Theorem~\ref{thm:easy}, meaning ${\rm null}\,G_{\rm
r}^{T}\cap{\rm null}\,G_{\rm d}^{T}={\rm span}\,\{\one_{q}\}$. This
would imply that $u$ belongs to the set ${\rm span}\,\{\one_{q}\}$,
which it doesn't. Therefore $\lambda>0$. Define $y=\Lambda_{\rm
r}G_{\rm r}^{T}u$. The vector $y$ is nonzero because if $y=0$ then
we would have the following contradiction $0=G_{\rm r}y=G_{\rm
r}\Lambda_{\rm r}G_{\rm r}^{T}u=R_{1}u=\lambda u\neq 0$. We can
write
\begin{eqnarray}\label{eqn:sgn}
\lambda y&=& \Lambda_{\rm r}G_{\rm r}^{T}(\lambda u)\nonumber\\
&=& \Lambda_{\rm r}G_{\rm r}^{T}R_{1}u\nonumber\\
&=& \Lambda_{\rm r}G_{\rm r}^{T}G_{\rm r}\Lambda_{\rm
r}G_{\rm r}^{T}u\nonumber\\
&=& \Lambda_{\rm r}G_{\rm r}^{T}G_{\rm r}y\,.
\end{eqnarray}
Since $\lambda>0$ and $\Lambda_{\rm r}$ is a diagonal matrix with
positive diagonal entries, \eqref{eqn:sgn} implies ${\rm
sgn}(y)={\rm sgn}(G_{\rm r}^{T}G_{\rm r}y)$. Furthermore, since
$D_{1}u=0$ implies $G_{\rm d}^{T}u=0$, we have $y\in{\rm
null}\,G_{\rm d}^{T}G_{\rm r}$ as can be seen from
\begin{eqnarray*}
G_{\rm d}^{T}G_{\rm r}y=G_{\rm d}^{T}G_{\rm r}\Lambda_{\rm r}G_{\rm
r}^{T}u=G_{\rm d}^{T}R_{1}u=\lambda G_{\rm d}^{T}u=0\,.
\end{eqnarray*}
To show the other direction, suppose now that there exists a nonzero
vector $x$ simultaneously satisfying ${\rm sgn}(x)={\rm sgn}(G_{\rm
r}^{T}G_{\rm r}x)$ and $G_{\rm d}^{T}G_{\rm r}x=0$. Let
$D_{2}\in{\rm lap}\,(\V,\,\B_{\rm d})$ be an arbitrary laplacian.
Define $v=G_{\rm r}x$. First note that $v\notin{\rm
span}\,\{\one_{q}\}$; for, otherwise, we would have $G_{\rm
r}^{T}v=0$ (since $G_{\rm r}^{T}\one_{q}=0$) and the following
contradiction would emerge
\begin{eqnarray*}
0={\rm sgn}(G_{\rm r}^{T}v)={\rm sgn}(G_{\rm r}^{T}G_{\rm r}x)={\rm
sgn}(x)\neq0
\end{eqnarray*}
because $x\neq 0$. Also, $G_{\rm d}^{T}v=G_{\rm d}^{T}G_{\rm r}x=0$,
yielding $D_{2}v=0$. Hence, we can write $v\in{\rm
null}\,D_{2}\setminus{\rm span}\,\{\one_{q}\}$. Let now $\Lambda$ be
a diagonal matrix with positive diagonal entries satisfying
$x=\Lambda G_{\rm r}^{T}G_{\rm r}x$. Such $\Lambda$ exists since
${\rm sgn}(x)={\rm sgn}(G_{\rm r}^{T}G_{\rm r}x)$. Define
$R_{2}=G_{\rm r}\Lambda G_{\rm r}^{T}$. Note that $R_{2}\in{\rm
lap}\,(\V,\,\B_{\rm r})$. Also, $v$ is an eigenvector of $R_{2}$
because we can write $R_{2}v=G_{\rm r}\Lambda G_{\rm r}^{T}G_{\rm
r}x=G_{\rm r}x=v$. Then by Lemma~\ref{fact:two} we have ${\rm
Re}\,\lambda_{2}(D_{2}+jR_{2})\leq 0$. Consequently, $\G$ is not
SSS.
\end{proof}

\section{Examples of SSS networks}

We begin this section by examining some example topologies from the
point of view of strong structural synchronization. Then we
establish certain generalizations valid for special classes of
interconnections. Henceforth, for simplicity, we consider triples
$(\V,\,\B_{\rm d},\,\B_{\rm r})$ with disjoint edge sets only, i.e.,
we let $\B_{\rm d}\cap\B_{\rm r}=\emptyset$. We emphasize that here
generality is not compromised in exchange for simplicity because
strong structural synchronization property of an arbitrary
interconnection can always be studied through one with disjoint edge
sets. More precisely:

\begin{proposition}
The triple $(\V,\,\B_{\rm d},\,\B_{\rm r})$ is SSS if and only if
$(\V,\,\B_{\rm d},\,\B_{\rm r}\setminus\B_{\rm d})$ is.
\end{proposition}

\begin{proof}
Given $\G=(\V,\,\B_{\rm d},\,\B_{\rm r})$ define
$\bar\G=(\V,\,\B_{\rm d},\,\bar\B_{\rm r})$ with $\bar\B_{\rm
r}=\B_{\rm r}\setminus\B_{\rm d}$. Note that $\G$ is SS when (and
only when) $\bar\G$ is SS. This follows from Theorem~\ref{thm:easy}
and the fact that $\B_{\rm d}\cup\B_{\rm r}=\B_{\rm
d}\cup\bar\B_{\rm r}$. Since structural synchronization is a
necessary condition for strong structural synchronization, we shall
focus on the case where the interconnections $\G$ and $\bar \G$ are
SS. Let $G_{\rm d}$, $G_{\rm r}$, and $\bar G_{\rm r}$ be the
incidence matrices of the graphs $(\V,\,\B_{\rm d})$, $(\V,\,\B_{\rm
r})$, and $(\V,\,\bar\B_{\rm r})$, respectively. Since $\bar\B_{\rm
r}\subset\B_{\rm r}$ and $\B_{\rm r}\setminus\bar\B_{\rm
r}\subset\B_{\rm d}$ we can find matrices $A,\,B$ and write $G_{\rm
r}=[\bar G_{\rm r}\ \, B]$ and $G_{\rm d}=[A\ \, B]$. When $B$ is
the empty matrix, the result trivially follows; for then we have
$G_{\rm r}=\bar G_{\rm r}$, yielding $\G=\bar\G$. Therefore we
henceforth let $B$ have at least one column. We now consider two
cases. {\em Case~1: $\bar\B_{\rm r}=\emptyset$.} That $\bar\B_{\rm
r}$ is empty has two immediate consequences. One of them is that
$\bar\G$ is SSS by Theorem~\ref{thm:SSS}, the other is $G_{\rm
r}=B$. Suppose now that $\G$ is not SSS. Then by
Theorem~\ref{thm:SSS} we can find a nonzero vector $x$
simultaneously satisfying $0=G_{\rm d}^{T}G_{\rm r}x=[A\ \,
B]^{T}Bx$ and $x\equiv G_{\rm r}^{T}G_{\rm r}x=B^{T}Bx$. The first
equation yields $B^{T}Bx=0$. Combining this with the second one
implies $x\equiv 0$, resulting in the contradiction $x=0$. Hence
$\G$ must be SSS. {\em Case~2: $\bar\B_{\rm r}\neq\emptyset$.}
Suppose $\G$ fails to be SSS. Then there exists a nonzero vector $x$
satisfying $G_{\rm d}^{T}G_{\rm r}x=0$ and $G_{\rm r}^{T}G_{\rm
r}x\equiv x$. Employing the partitioning $x=[x_{1}^{T} \
x_{2}^{T}]^{T}$ we can rewrite these equations as
\begin{eqnarray}
\left[\begin{array}{cc}A^{T}\bar G_{\rm r}&A^{T}B\\ B^{T}\bar G_{\rm r}&B^{T}B\end{array}\right]
\left[\begin{array}{c}x_{1}\\x_{2}\end{array}\right]&=&0\,,\label{eqn:belikewater1}\\
\left[\begin{array}{cc}\bar G_{\rm r}^{T}\bar G_{\rm r}&\bar G_{\rm r}^{T}B\\
B^{T}\bar G_{\rm
r}&B^{T}B\end{array}\right]
\left[\begin{array}{c}x_{1}\\x_{2}\end{array}\right]&\equiv&\left[\begin{array}{c}x_{1}\\x_{2}\end{array}\right]\,.\label{eqn:belikewater2}
\end{eqnarray}
Now, \eqref{eqn:belikewater1} yields $B^{T}\bar G_{\rm
r}x_{1}+B^{T}Bx_{2}=0$ while $B^{T}\bar G_{\rm
r}x_{1}+B^{T}Bx_{2}\equiv x_{2}$ by \eqref{eqn:belikewater2}.
Combining these we obtain $x_{2}=0$. Since $x$ was nonzero, we have
to have $x_{1}\neq 0$. Moreover, in the light of $x_{2}=0$, the
equations \eqref{eqn:belikewater1} and \eqref{eqn:belikewater2} can
be reduced to
\begin{eqnarray*}
0&=&[A\ \ B]^{T}\bar G_{\rm r}x_{1}=G_{\rm d}^{T}\bar G_{\rm
r}x_{1}\,,\\
x_{1}&\equiv&\bar G_{\rm r}^{T}\bar G_{\rm r}x_{1}\,.
\end{eqnarray*}
Thence we deduce by Theorem~\ref{thm:SSS} that $\bar\G$ cannot be
SSS because $x_{1}$ is nonzero. It is not difficult to see that the
steps we have taken can be traced back. That is, SSS of $\bar G$
implies SSS for $\G$. The proof is therefore complete.
\end{proof}
\vspace{0.12in}

\noindent {\bf Example~1.} As our first example, let us recall the
four-node interconnection (for which we let
$\V=\{\nu_{1},\,\nu_{2},\,\nu_{3},\,\nu_{4}\}$ be the vertex set) we
visited earlier in the paper, reproduced in Fig.~\ref{fig:graph1},
where the resistor represents the (only) edge in the set $\B_{\rm
d}=\{\{\nu_{1},\,\nu_{3}\}\}$ and the inductors represent by the
edges in $\B_{\rm
r}=\{\{\nu_{1},\,\nu_{2}\},\,\{\nu_{2},\,\nu_{3}\},\,\{\nu_{3},\,\nu_{4}\}\}$.
By Theorem~\ref{thm:easy} this interconnection $(\V,\,\B_{\rm
d},\,\B_{\rm r})$ is clearly SS. To check whether it is also SSS,
let us apply the test presented in Theorem~\ref{thm:SSS}.
\begin{figure}[h]
\begin{center}
\includegraphics[scale=0.4]{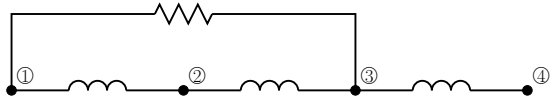}
\caption{An interconnection that is not SSS. The edges in $\B_{\rm
d}$ are depicted as resistors; the edges in $\B_{\rm r}$ are
depicted as inductors.}\label{fig:graph1}
\end{center}
\end{figure}
Note that the incidence matrices associated to $(\V,\,\B_{\rm
d},\,\B_{\rm r})$ read
\begin{eqnarray*}
G_{\rm d}=\left[\begin{array}{r}1\\ 0\\ -1\\ 0\end{array}\right]\,
,\qquad G_{\rm r}=\left[\begin{array}{rrr}1&0&0\\-1&1&0\\0&-1&1\\
0&0&-1\end{array}\right]\,.
\end{eqnarray*}
Choosing, for instance, the vector $x=[2\ -1\ 1]^{T}$ yields $G_{\rm
r}x=[2\ -3\ \,2\ -1]^{T}$, whence follows $G_{\rm r}^{T}G_{\rm
r}x=[5\ -5\ 3]^{T}$ and, consequently, $G_{\rm r}^{T}G_{\rm
r}x\equiv x$. Moreover, we have $G_{\rm d}^{T}G_{\rm r}x=0$.
Therefore the interconnection is not SSS. \vspace{0.12in}

A somewhat physical interpretation of the above equations is
possible in the following way. Let the $i$th column of the incidence
matrix $G_{\rm r}$ be $[e_{k}-e_{\ell}]$. Then if the $i$th entry of
the vector $x$ is $c_{k\ell}$, we can say a {\em current} of value
$c_{k\ell}$ flows through the edge
$\{\nu_{k},\,\nu_{\ell}\}\in\B_{\rm r}$ from $\nu_{k}$ to
$\nu_{\ell}$; or, equivalently, a current of value $-c_{k\ell}$
flows from $\nu_{\ell}$ to $\nu_{k}$. That is, the direction of any
current can be reversed by changing the sign of its value. This
allows us, without loss of generality, to consider only the case
where all currents are nonnegative. As for the edges that belong to
$\B_{\rm d}$, we assign them zero currents, which makes their
directions immaterial. Moreover, the mapping $x\mapsto G_{\rm r}x$
can be interpreted as that the edge currents ($x$) generate the
vertex {\em potentials} ($G_{\rm r}x$) through the rule:
\begin{itemize}
\item[\bf (A1)] The potential of a vertex equals the sum of outgoing currents
minus the sum of incoming currents associated to that vertex.
\end{itemize}

\begin{remark}\label{rem:zerosum}
A direct implication of (A1) is that the vertex potentials
throughout any interconnection always add up to zero.
\end{remark}

Having defined edge currents and vertex potentials now we can
interpret the constraints $G_{\rm d}^{T}G_{\rm r}x=0$ and $G_{\rm
r}^{T}G_{\rm r}x\equiv x$ easily. The former means:
\begin{itemize}
\item[\bf (A2)] The potentials of any two vertices that are connected through an
edge that belongs to $\B_{\rm d}$ are equal.
\end{itemize}
Whereas $G_{\rm r}^{T}G_{\rm r}x\equiv x$ is equivalent to:
\begin{itemize}
\item[\bf (A3)] If two vertices with different potentials are
connected by an edge then the current on that edge is positive and
flows from higher potential to lower potential. And if two vertices
have the same potential then the current through the edge that
connects them is zero.
\end{itemize}
The edge currents ($[2\ -1\ 1]^{T}=x$) and the vertex potentials
($[2\ -3\ \,2\ -1]^{T}=G_{\rm r}x$) we have used in our example are
shown in Fig.~\ref{fig:graph4}, where we observe that all three
conditions (A1), (A2), (A3) are satisfied. Let us now formalize our
observations.
\begin{figure}[h]
\begin{center}
\includegraphics[scale=0.4]{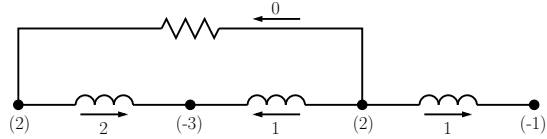}
\caption{The vertex potentials (shown in brackets) and edge currents
of Example~1.}\label{fig:graph4}
\end{center}
\end{figure}

\begin{definition}
Given an SS interconnection $(\V,\,\B_{\rm d},\,\B_{\rm r})$ let
$\D$ denote some {\em (compatible) distribution} where each edge is
assigned a (nonnegative) current and each vertex a potential;
obeying the rules (A1), (A2), and (A3). The distribution $\D$ is
said to be {\em nontrivial} if it has at least one nonzero current,
otherwise it is called {\em trivial}.
\end{definition}

This definition lets us make a restatement of Theorem~\ref{thm:SSS}:

\begin{theorem}\label{thm:SSS2}
An SS interconnection is SSS if and only if it does not admit a
nontrivial distribution.
\end{theorem}

\begin{figure}[h]
\begin{center}
\includegraphics[scale=0.4]{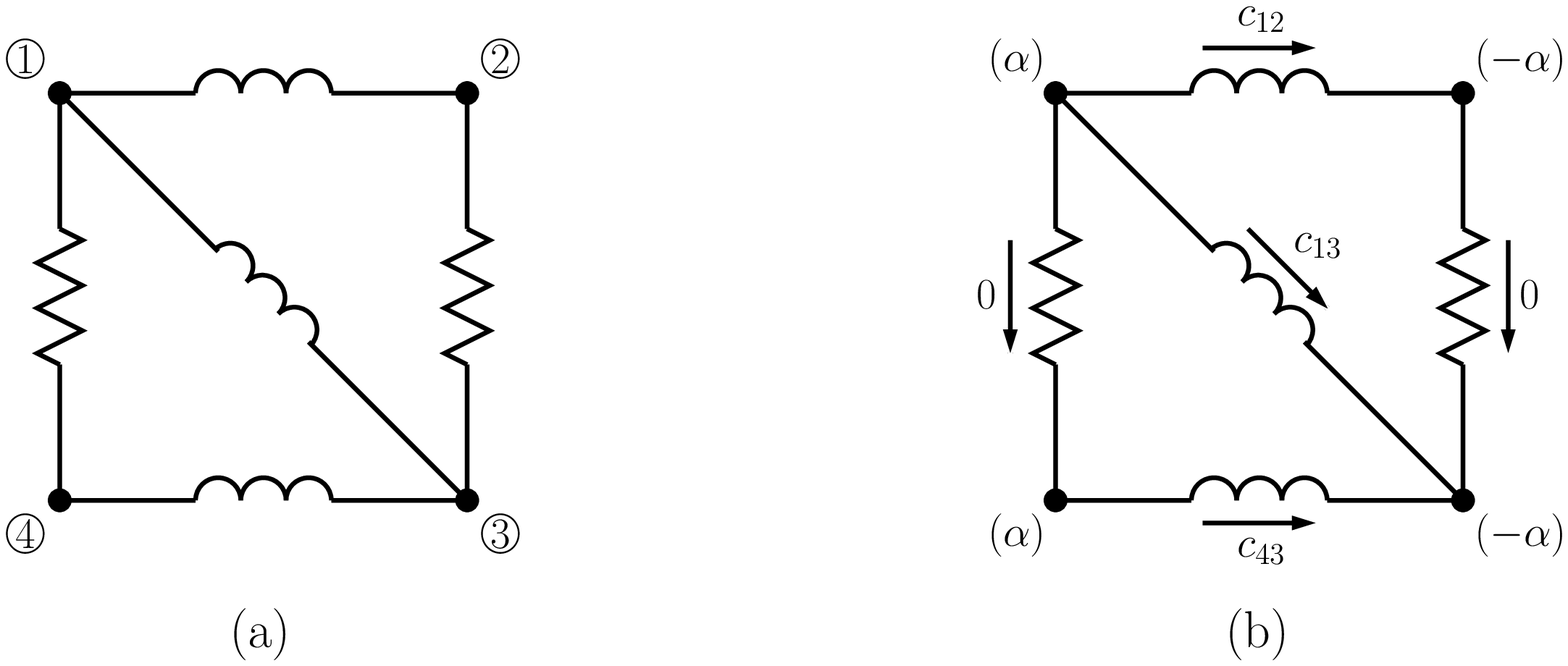}
\caption{An SSS interconnection. (The edges in $\B_{\rm d}$ are
depicted as resistors; the edges in $\B_{\rm r}$ are depicted as
inductors.)}\label{fig:graph2}
\end{center}
\end{figure}

\noindent{\bf Example~2.} This time consider the interconnection
shown in Fig.~\ref{fig:graph2}a. Let us analyze it using our last
theorem. Let $\D$ be an arbitrary distribution, where the potential
$p_{1}$ of the vertex ${\textcircled{\raisebox{-.9pt} {1}}}$ is
$\alpha$. Then by (A2) we have $p_{4}=p_{1}=\alpha$, whence follows,
by Remark~\ref{rem:zerosum} and (A2), $p_{2}=p_{3}=-\alpha$. These
vertex potentials are shown in Fig.~\ref{fig:graph2}b. Let us now
focus on the edge currents $c_{12},\,c_{13},\,c_{43}$ shown in the
same figure. Applying rule (A1) on the second and fourth vertices
yields $c_{12}=\alpha$ and $c_{43}=\alpha$, respectively. Finally,
we apply it on the first vertex and obtain $c_{12}+c_{13}=\alpha$,
whence we deduce $c_{13}=0$. Now, that the current $c_{13}$ is zero
implies by (A3) that the first and third vertices must have the same
potential $p_{1}=p_{3}$, i.e., $\alpha=-\alpha$, yielding
$\alpha=0$. This means that the distribution $\D$ is trivial.
Theorem~\ref{thm:SSS2} then tells us that the interconnection is
SSS. \vspace{0.12in}

Further examples are given in Fig.~\ref{fig:graph3}, where a
nontrivial current distribution (from which, by (A1), the potential
distribution can be determined) is provided for the instances that
are not SSS. In the light of Theorem~\ref{thm:SSS2} we next provide
certain graph theoretical conditions that guarantee SSS for special
types of topologies; namely, path graphs, cycles, and trees. But
this requires a quick review of some graph terminology first.

\begin{figure}[h]
\begin{center}
\includegraphics[scale=0.4]{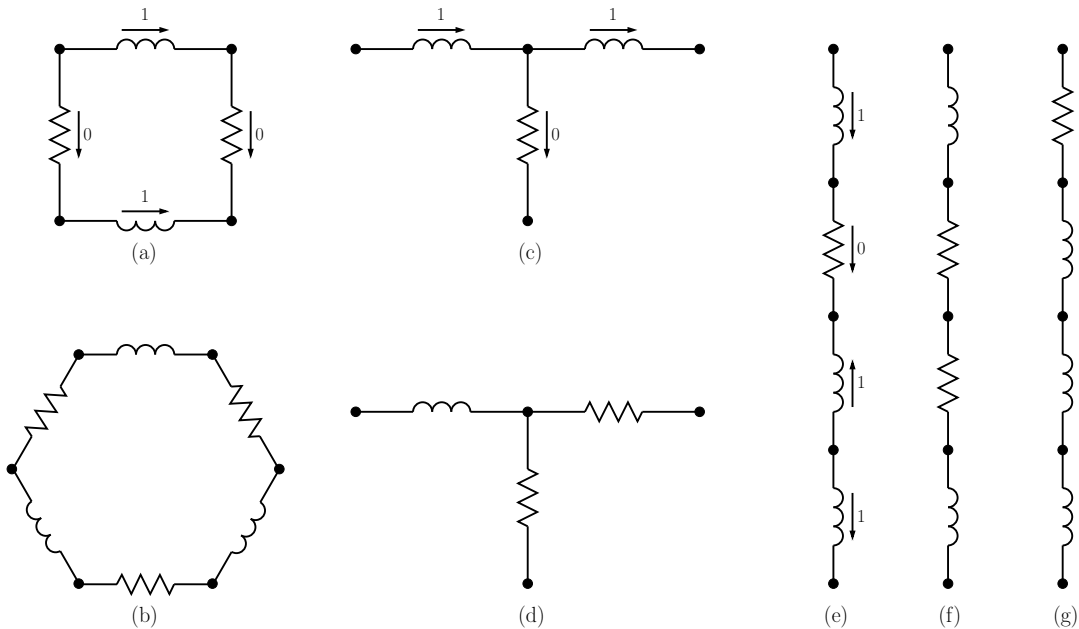}
\caption{Various interconnections. The topologies (b), (d), (f), and
(g) are SSS; whereas (a), (c), and (e) are not SSS. (The edges in
$\B_{\rm d}$ are depicted as resistors; the edges in $\B_{\rm r}$
are depicted as inductors.)}\label{fig:graph3}
\end{center}
\end{figure}

Let $\Gamma=(\V,\,\B)$ be a graph with the vertex set
$\V=\{\nu_{1},\,\nu_{2},\,\ldots,\,\nu_{q}\}$ and $q\geq 2$. The
graph $\Gamma$ is a {\em path} if the edge set can be written as
$\B=\{\{\nu_{1},\,\nu_{2}\},\,\{\nu_{2},\,\nu_{3}\},\,\ldots,\,\{\nu_{q-1},\,\nu_{q}\}\}$
perhaps after a relabeling of the vertices; it is a {\em cycle} if
$\B=\{\{\nu_{1},\,\nu_{2}\},\,\{\nu_{2},\,\nu_{3}\},\,\ldots,\,\{\nu_{q-1},\,\nu_{q}\},\,\{\nu_{q},\,\nu_{1}\}\}$.
Namely, the incidence matrices of a path and a cycle have the
following structures
\begin{eqnarray*}
G_{\rm path}=\left[\begin{array}{rrrrr} 1&0&0&\cdots&0\\
-1&1&0&\cdots&0\\
0&-1&1&\cdots&0\\
\vdots&\vdots&\vdots&\ddots&\vdots\\
0&0&0&\cdots&1\\
0&0&0&\cdots&-1
\end{array}\right]_{q\times(q-1)}\,,\qquad
G_{\rm cycle}=\left[\begin{array}{rrrrrr} 1&0&0&\cdots&0&-1\\
-1&1&0&\cdots&0&0\\
0&-1&1&\cdots&0&0\\
\vdots&\vdots&\vdots&\ddots&\vdots&\vdots\\
0&0&0&\cdots&1&0\\
0&0&0&\cdots&-1&1
\end{array}\right]_{q\times q}\,.
\end{eqnarray*}
Note that the $ij$th entry of $G_{\rm path}$ reads
\begin{eqnarray*}
[G_{\rm path}]_{ij}=\left\{\begin{array}{rl} -1&\mbox{for}\ \
i=j+1\,,\\ 1&\mbox{for}\ \ i=j\,,\\ 0&\mbox{elsewhere}\,.
\end{array}\right.
\end{eqnarray*}
 A {\em tree} is a
generalization of path in the sense that its incidence matrix
$G_{\rm tree}$ satisfies (perhaps after a suitable relabeling)
\begin{eqnarray*}
[G_{\rm tree}]_{ij}=\left\{\begin{array}{cl} -1&\mbox{for}\ \ i=j+1\,,\\
0\ \mbox{or}\ 1&\mbox{for}\ \ i\leq j\,,\\ 0&\mbox{elsewhere}\,,
\end{array}\right.
\end{eqnarray*}
in addition to the usual constraint that each column is of the form
$[e_{k}-e_{\ell}]$. A vertex $\nu_{i}$ of a tree is called a {\em
leaf} if the $i$th row of the incidence matrix contains only a
single nonzero entry. For instance, a path has exactly two leaves:
first and the last vertices. Henceforth for a given $(\V,\,\B_{\rm
d},\,\B_{\rm r})$ we denote by $[\B_{\rm r}]\subset\V$ the set of
vertices that are associated to the edges in $\B_{\rm r}$. That is,
$[\B_{\rm r}]=\{\nu\in\V:\nu\in\{\nu_{i},\,\nu_{j}\}\in\B_{\rm
r}\}$. The set $[\B_{\rm d}]$ is defined similarly. We now list some
straightforward consequences of Theorem~\ref{thm:SSS2}.

\begin{corollary}\label{cor:path}
Let the interconnection $\G=(\V,\,\B_{\rm d},\,\B_{\rm r})$ be such
that $\B_{\rm d}\cap\B_{\rm r}=\emptyset$ and the graph
$(\V,\,\B_{\rm d}\cup\B_{\rm r})$ is a path. Then $\G$ is SSS if and
only if $\V\setminus[\B_{\rm r}]$ is not empty.
\end{corollary}

\begin{corollary}\label{cor:cycle}
Let the interconnection $\G=(\V,\,\B_{\rm d},\,\B_{\rm r})$ be such
that $\B_{\rm d}\cap\B_{\rm r}=\emptyset$ and the graph
$(\V,\,\B_{\rm d}\cup\B_{\rm r})$ is a cycle with $q$ nodes. Then
$\G$ is SSS if and only if
\begin{enumerate}
\item Either $\V\setminus[\B_{\rm r}]$ is not empty,
\item Or else $\V\setminus[\B_{\rm d}]$ is empty and $q/2$ is odd.
\end{enumerate}
\end{corollary}

\begin{corollary}\label{cor:tree}
Let the interconnection $\G=(\V,\,\B_{\rm d},\,\B_{\rm r})$ be such
that $\B_{\rm d}\cap\B_{\rm r}=\emptyset$ and the graph
$\Gamma=(\V,\,\B_{\rm d}\cup\B_{\rm r})$ is a tree. Then $\G$ is SSS
if $[\B_{\rm r}]$ contains at most one leaf of $\Gamma$.
\end{corollary}

Let us revisit the interconnections in Fig.~\ref{fig:graph3} in the
light of these corollaries. The interconnections (e), (f), and (g)
are paths for which the set $\V\setminus[\B_{\rm r}]$ is empty only
for (e). By Corollary~\ref{cor:path} therefore the path (e) is not
SSS, while the paths (f) and (g) are. The interconnections (a) and
(b) are cycles, and clearly $\V\setminus[\B_{\rm r}]$ is empty for
both. In this case Corollary~\ref{cor:cycle} advises us to check
$\V\setminus[\B_{\rm d}]$, which too is empty for both. It is not
difficult to see that when $\V\setminus[\B_{\rm r}]$ and
$\V\setminus[\B_{\rm d}]$ are simultaneously empty for a cycle it
must be that the sets $\B_{\rm r}$ and $\B_{\rm d}$ must contain
equal number of edges. This means that the number of nodes $q$ is
necessarily an even number, yielding that $q/2$ is an integer. For
(a) we have $q/2=2$, an even number, while for (b) $q/2=3$ is odd.
Via Corollary~\ref{cor:cycle} we arrive therefore at the conclusion
that only the cycle (b) is SSS. Finally, consider the trees (c) and
(d). The tree (c) has two leaves that belong to $[\B_{\rm r}]$. In
such a case we cannot use Corollary~\ref{cor:tree} since the
condition given there is only sufficient. For the other tree,
however, the set $[\B_{\rm r}]$ contains a single leaf. Hence the
interconnection (d) must be SSS.

\section{Conclusion}

In this paper we presented a structural analysis of synchronization
in linear networks of identical oscillators (e.g. LC circuits)
coupled through both dissipative connectors (e.g. resistors) and
restorative connectors (e.g. inductors). We provided answers to the
following two questions. First, for a given coupling structure, when
can we find a suitable set of coupling strengths that guarantees
asymptotic synchronization of the oscillators? Second, for what type
of structures is synchronization guaranteed regardless of the
coupling strengths? The answer to the first question turned out to
be very simple: A suitable choice of parameter values (yielding
synchronization) exists when the network is connected and is not
entirely free of dissipative coupling. The second problem however
admitted only a more elaborate solution, which required us to
introduce flow diagrams (defined through three relatively
nontechnical rules) for our analysis. This solution yielded simple
conditions on synchronization for networks whose coupling topology
is either a path or a cycle or a tree.

\bibliographystyle{plain}
\bibliography{references}
\end{document}

%% file: macros.tex
\usepackage[final]{graphicx}
\usepackage{psfrag}
\usepackage{amsmath,amsfonts,amssymb,amsxtra,subeqnarray,bm}

\newlength{\extramargin}
\newcommand{\Real}{\ensuremath{{\mathbb{R}}}}

\newcommand{\Complex}{\ensuremath{{\mathbb{C}}}}

\newcommand{\C}{\ensuremath{\mathcal C}}

\newcommand{\D}{\ensuremath{\mathcal D}}

\newcommand{\V}{\ensuremath{\mathcal V}}

\newcommand{\G}{\ensuremath{\mathcal G}}

\newcommand{\B}{\ensuremath{\mathcal B}}

\newcommand{\one}{\ensuremath{{\mathbf{1}}}}

\newtheorem{theorem}{Theorem}

\newtheorem{corollary}{Corollary}

\newtheorem{lemma}{Lemma}

\newtheorem{definition}{Definition}
\newtheorem{remark}{Remark}

\newtheorem{proposition}{Proposition}
\newenvironment{proof}{\noindent {\bf Proof.}}{\hfill \hspace*{1pt}\hfill$\blacksquare$}